\documentclass[12pt]{amsart}
\usepackage{graphicx}
\usepackage{amsthm}
\usepackage{tikz}
\usepackage{mathrsfs}
\usepackage{url}
\usetikzlibrary{calc,decorations.markings}

\makeatletter
\@namedef{subjclassname@2010}{%
  \textup{2010} Mathematics Subject Classification}
\makeatother

\newtheorem{thm}{Theorem}[section]

\newtheorem{defi}{Definition}[section]
\newtheorem{rmk}{Remark}[section]

\newtheorem{note}{Note}[section]

\newtheorem{asump}{Assumptions}[section]

\def\A{\mathcal{A}}
\def\half{\frac{1}{2}}
\def\Re{\mathrm{Re}}

\def\res{\mathrm{Res}}
\def\Co{\mathscr{C}}

\def\R{\mathbb{R}}
\def\N{\mathbb{N}}

\def\d{\mathrm{d}}

\def\A{\mathcal{A}}
\def\M{\mathcal{M}}

\def\twopi{\frac{1}{2\pi i}}
\def\B{\mathfrak B}
\def\m{\mathfrak{m}}

\begin{document}


\baselineskip=17pt


\title[Measure theoretic aspects of error terms]{Measure theoretic aspects of oscillations of error terms }

\author[K. Mahatab]{Kamalakshya Mahatab}
\address{Institute of Mathematical Sciences, HBNI, 
CIT Campus, Taramani, Chennai 600113, India}
\address{NTNU, Trondheim 7491, Norway}
\email[K. Mahatab]{accessing.infinity@gmail.com}
\author[A. Mukhopadhyay]{Anirban Mukhopadhyay}
\address{Institute of Mathematical Sciences, HBNI, 
CIT Campus, Taramani, Chennai 600113, India}
\email[A. Mukhopadhyay]{anirban@imsc.res.in}

\date{}

\begin{abstract}
We consider fluctuations of error terms $\Delta(x)$ appearing in the asymptotic formula
for a summatory function of coefficients of the Dirichlet series. 
These are quantified via $\Omega$ and $\Omega_{\pm}$ estimates. We obtain $\Omega$ 
bounds for Lebesgue measure of the sets   
\[\{T\leq x \leq 2T: \Delta(x)>\lambda x^{\alpha}\} \text{ and } 
\{T\leq x \leq 2T: \Delta(x)< -\lambda x^{\alpha}\}\]
for some $\alpha, \lambda>0$.
Primary aim of this article is to develop a general framework to approach these problems. 
We rediscover several classical results in general setting with weak assumptions.
Moreover, several applications of these
methods have been discussed and new results have been obtained for some Dirichlet series. 
\end{abstract}



\maketitle


\section{Introduction}
 Analysis of error terms in asymptotic formulas is of considerable importance in various fields of mathematics. 
 For example, consider the von-Mangoldt function 
 \[\Lambda(n)=\begin{cases}
                   \log p \quad &\mbox{ if } \ n=p^r, \ r\in \mathbb{N}, \text{ and } p \text{ prime,}\\
                   0 & \mbox{ otherwise .}
                   \end{cases}
 \]
The Prime Number Theorem says that
 \[\sum_{n\leq x}\Lambda(n)=x + \Delta(x),\]
where $\Delta(x)$ is $o(x)$. It is also known that the famous Riemann Hypothesis is equivalent to 
(see \cite{PNT_Under_RH}, also Theorem~\ref{thm:landu_omegapm} below )
\begin{equation*}
 \Delta(x)=O\left(x^{\half}\log^2 x\right).
\end{equation*}
The following result, proved by 
Hardy and Littlewood \cite{HardyLittlewoodPNTOmegapm}, shows that such an upper bound for 
$\Delta(x)$ is optimal in terms of the power of $x$:
 \begin{align*}
 \limsup \frac{\Delta(x)}{x^\half\log\log\log x} > 0 \quad \text{ and } 
 \quad \liminf \frac{\Delta(x)}{x^\half\log\log\log x} < 0.
\end{align*}
A weaker result by Landau \cite{Landau} gives
\begin{align*}
 \limsup \frac{\Delta(x)}{x^\half} > 0 \quad \text{ and } \quad \liminf \frac{\Delta(x)}{x^\half} < 0.
\end{align*}
However, Landau's method has wide applications, and it is flexible to give some measure theoretic results.
This method crucially depends on existence of a complex pole of $\zeta'(s)/\zeta(s)$
with real part $\half$.
In this paper, we shall investigate on a 
quantitative version of Landau's result by obtaining the Lebesgue measure of the sets 
where $\Delta(x)>\lambda x^{1/2}$ and 
$\Delta(x)<-\lambda x^{\half}$, for some $\lambda>0$. 
\subsection*{Outline}
In general, consider a sequence of real numbers $\{a_n\}_{n=1}^{\infty}$ having Dirichlet series
\begin{equation}\label{def:D}
 D(s)=\sum_{n=1}^{\infty}\frac{a_n}{n^s}.
\end{equation}
that converges in some half-plane. 
The Perron summation formula \cite[II.2.1]{TenenAnPr} uses analytic properties of $D(s)$ to give
\begin{equation}\label{def:M}
\sum^*_{n\leq x}a_n=\M(x)+\Delta(x), 
\end{equation}
where $\M(x)$ is the main term, $\Delta(x)$ is the error term ( which would be specified later ) 
and $\sum^*$ is defined as 
\begin{equation*}
\sum^*_{n\leq x} a_n =
\begin{cases}
\sum_{n\leq x} a_n \ & \text{if } x\notin \mathbb N \\
\sum_{n< x} a_n +\half a_x  \ &\text{if } x\in\mathbb N.
\end{cases}
\end{equation*}

In this paper, we obtain $\Omega$ and $\Omega_{\pm}$ results to measure fluctuations 
of $\Delta(x)$. We also obtain $\Omega$ bounds for Lebesgue measure of the sets on which 
these fluctuations occur. This approach requires meromorphic continuation of the 
Mellin transform $A(s)$ of $\Delta(x)$ 
which is defined as:
 \[A(s)=\int_{1}^{\infty}\frac{\Delta(x)}{x^{s+1}}\d x\]
for a complex variable $s$.
In general, $A(s)$ is holomorphic in some half plane.
In Section~\ref{sec:analytic_continuation}, we shall discuss meromorphic continuation of $A(s)$. 

In Section \ref{sec:preliminaries}, we revisit Landau's method and obtain some measure theoretic results.
If $A(s)$ has a singularity at $\sigma_0+it_0$ for some $t_0\neq 0$, and has no singularity on real line
for $\Re(s)\ge \sigma_0$,
then Landau's method gives 
\[ \Delta(x)=\Omega_{\pm} (x^{\sigma_0}). \]
A measure theoretic version of Landau's method was first used by Kaczorowski and Szyd\l o \cite{KaczMeasure} on $E_2(x)$, where
\[ \int_0^x \left|\zeta\left(\half+it \right)\right|^4 \d t = x P(\log x) + E_2(x)\]
and $P$ being a certain polynomial of degree $4$. Ivi\'c and Motohashi \cite{Ivic_moto} proved that 
\[ E_2(x) \ll x^{2/3}(\log x)^c \ \text{ for some }\ c>0, \]
and further in \cite{Motohashi2} Motohashi showed that 
\[ E_2(x)=\Omega_{\pm}(\sqrt x). \]
The result of Kaczorowski and Szyd\l o mentioned above says that 
there exist constants $\lambda_0, \nu>0$ such that 
\begin{align*}
 &&\m\{1\le x\le T: E_2(x)>\lambda_0\sqrt x \} = \Omega(T/{(\log T)^{\nu}})& \\
 &\text{ and }&\m\{1\le x\le T: E_2(x)<-\lambda_0\sqrt x \} = \Omega(T/{(\log T)^{\nu}})& \ 
 \text{ as } T\rightarrow \infty,
\end{align*}
and where $\m$ is the Lebesgue measure 
\footnote{Throughout this paper, $\m$ will denote the Lebesgue measure on 
the real line $\mathbb R$.}.
These results not only prove $\Omega_{\pm}$ bounds, 
but also give quantitative estimates for the occurrences of such fluctuations. 
Let
\begin{align*}
&\sum_{n\leq x} G_k(n) = \frac{x^k}{k!} -k \sum_\rho \frac{x^{k-1+\rho}}{\rho(1+\rho)\cdots(k-1+\rho)} 
+ \Delta_k(x), 
\end{align*}
where the Goldbach numbers $G_k(n)$ are defined as
\[\quad G_k(n)=\sum_{\substack{n_1,\ldots n_k \\ n_1+\cdots+n_k=n}}\Lambda(n_1)\cdots\Lambda(n_k),\]
and $\rho$ runs over nontrivial zeros of the Riemann zeta function $\zeta(s)$.
Bhowmik, Ramar\'e and Schlage-Puchta proved that under Riemann Hypothesis
\begin{align*}
&&\m\{T\le x\le 2T: \Delta_k(x)>(\mathfrak c_k + \mathfrak c_k')x^{k-1} \} 
= \Omega(T/{(\log T)^{c}}) \\
&\text{ and }&\m\{T\le x\le 2T: \Delta_k(x)<(\mathfrak c_k-\mathfrak c_k') x^{k-1} \} 
= \Omega(T/{(\log T)^{c}}),
\end{align*}
where $c$ is a computable positive constant, $k\geq 2$, $\mathfrak{c}_k$ and $\mathfrak{c}_k'$ are real numbers that depend on $k$ and $\mathfrak{c}_k'>0$.
In this paper, we obtain analogous results for other functions. 
In Theorem~\ref{thm:omega_pm_main}, we further generalize 
this method to have more applications. 

\subsection*{Applications}
We conclude the introduction to this paper by mentioning a few applications. 
\subsubsection*{\textbf{ Twisted Divisor Function}}
For a fixed $\theta\neq 0$,  we consider 
\begin{equation}\label{eq:tau-n-theta_def}
 \tau(n, \theta)=\sum_{d\mid n}d^{i\theta}\ .
\end{equation}
This function is used in \cite[Chapter 4]{DivisorsHallTenen} to measure the clustering of divisors.
The Dirichlet series of $|\tau(n, \theta)|^2$ can be expressed in terms of the Riemann zeta function as 
\begin{equation}\label{eq:dirichlet_series_tauntheta}
 D(s)=\sum_{n=1}^{\infty}\frac{|\tau(n, \theta)|^2}{n^s}=\frac{\zeta^2(s)\zeta(s+i\theta)\zeta(s-i\theta)}{\zeta(2s)}
 \quad\quad  \text{for}\quad \Re(s)>1.
\end{equation}
In \cite[Theorem 33]{DivisorsHallTenen}, Hall and Tenenbaum proved that
\begin{equation}\label{eq:formmula_tau_ntheta}
 \sum_{n\leq x}^*|\tau(n, \theta)|^2=\omega_1(\theta)x\log x + \omega_2(\theta)x\cos(\theta\log x)
+\omega_3(\theta)x + \Delta(x),
\end{equation}
where $\omega_i(\theta)$s are explicit constants depending only on $\theta$ and 
\begin{equation}\label{eq:upper_bound_delta}
 \Delta(x)=O_\theta(x^{1/2}\log^6x).
\end{equation}
Here the main term comes from the residues of  $D(s)$ at $s=1, 1\pm i\theta $.
All other poles of $D(s)$ come from the zeros of $\zeta(2s)$. Using a pole on the line $\Re(s)=1/4$, 
Landau's method gives
\[\Delta(x)=\Omega_{\pm}(x^{1/4}).\]
In order to apply the method of Kaczorowski and Szyd\l o, we need
\[ \int_T^{2T} \Delta^2(x) \d x \ll T^{2\sigma_0+1+\epsilon}, \]
for any $\epsilon >0$ and  $\sigma_0=1/4$; however, such an estimate is not possible \cite[V.4]{thesis1}.
Generalization of this method in Theorem~\ref{thm:omega_pm_main} can be applied to get
\begin{align*}
 \m \left( \A_j\cap [T, 2T]\right)=\Omega\left(T^{1/2}(\log T)^{-12}\right) \quad \text{ for } j=1, 2,
\end{align*}
and here $\A_j$s' for $\Delta(x)$ are defined as
\begin{align*}
 &&\A_1&=\left\{x: \Delta(x)>(\lambda(\theta)-\epsilon)x^{1/4}\right\}&\\
 &\text{and}&\A_2&=\left\{ x : \Delta(x)<(-\lambda(\theta)+\epsilon)x^{1/4}\right\},&\\
\end{align*}
for any $\epsilon>0$ and $\lambda(\theta)>0$ as in (\ref{eqn:lambda_theta}). But under Riemann Hypothesis, the above $\Omega$ bounds can be improved to
\begin{align*}
\m\left(\A_j\right) =\Omega\left(T^{3/4-\epsilon}\right),\quad \text{ for } j=1, 2
\end{align*}
and for any $\epsilon>0$.
\subsubsection*{\textbf{ The M\"obius Function }}
The  M\"obius function $\mu$ is supported on squarefree integers, which is multiplicative and takes the value 
$-1$ on primes. Let
\[\Delta(x)=\sum_{n\leq x}^*\mu(n).\]
Under Riemann Hypothesis \cite{sound_mobius} \cite{balazard}, 
\[\Delta(x)\ll \sqrt x \exp\left((\log x)^\half(\log\log x)^{\frac{5}{2}+\epsilon}\right)\ \text{ for any }\epsilon>0.\]
The following unconditional upper bound for $\Delta(x)$ can be computed using the zero free region of $\zeta(s)$:
\[\Delta(x)\ll x \exp\left(-c(\log x)^\half(\log\log x)^{-\frac{1}{5}}\right)\ \text{ for some } c>0.\]
Assuming Riemann Hypothesis and linear independence of nontrivial zeros of the Riemann Zeta function, 
Ingham \cite{ingham_mobius} proved that 
\[\limsup \frac{\Delta(x)}{x^\half}\rightarrow\infty \text{ and } \liminf \frac{\Delta(x)}{x^\half}\rightarrow-\infty \text{ as } x\rightarrow\infty.\]
However, an unconditional proof of the Ingham's theorem is still open. Below we measure the sets where amplitude
of $\Delta(x)$ is at least of order $x^\half$.

We define
\begin{align*}
 \A_1&=\left\{x: \Delta(x)>(\lambda-\epsilon_0)x^{1/2}\right\},\\
 \A_2&=\left\{ x : \Delta(x)<(-\lambda+\epsilon_0)x^{1/2}\right\},\\
\end{align*}
where $\lambda>0$ be as in Section~\ref{subsec:mobius} and $0<\epsilon_0<\lambda_2$. 
If we assume Riemann Hypothesis, then the theorem of Kaczorowski and Szyd\l o  (see Theorem~\ref{thm:kaczorowski_correct} below)
 gives 
 \[ \m\left(\A_j\cap[T, 2T]\right)=\Omega\left(T \exp\left(-(\log x)(\log\log x)^{5+\epsilon}\right)\right) \text{ for } j=1, 2.\]
However, as an application of Theorem~\ref{thm:omega_pm_main}, we prove the following unconditional weaker bound:
\[ \m\left(\A_j\cap[T, 2T]\right)=\Omega\left(T^{1-\epsilon}\right), \]
for $j=1, 2$ and for any $\epsilon>0$.
\subsubsection*{\textbf{M\"obius Function of Beurling Primes}}
A set of generalized Beurling primes is a sequence of positive real numbers 
$\B =\{p_j\}_{j=1}^{\infty}$ such that $1<p_j\leq p_{j+1}$. The generalized Beurling integers 
${n_k}$ associated to the sequence $\B$ is a multiset of all possible finite products of the form
$p_{\nu_1}^{\alpha_1}\ldots p_{\nu_m}^{\alpha_m}$ with $\nu_j<\nu_{j+1}$ and $\alpha_j\in \N$.
Define 
$$N_\B:=\sum_{n_j\leq x}1 \ \text{ and } \ \pi_\B=\sum_{p_j\leq x} 1.$$
Beurling \cite{beurling} proved that if 
\[N_\B(x)=ax+O\left(\frac{x}{\log^\gamma x}\right) \ \text{ for } \gamma> 3/2, a>0
\text{ and as } x\rightarrow\infty,\]
then
\[\pi_\B(x)\sim\frac{x}{\log x}.\]
The proof of Beurling relies on establishing analytic properties of the Beurling zeta function $\zeta_\B(s)$ defined as follows:
\[\zeta_\B(s)=\prod_{j=1}^{\infty}\left(1-p_j^{-s}\right)^{-1}.\]
Diamond, Montgomery and Vorhauer \cite{diamond} showed that there are Beurling zeta functions having zeros 
very close to the line  $1+it$ and for which the Prime Number Theorem holds. This is counter intuitive, given that 
the nontrivial zeros of the Riemann zeta function are conjectured to lie on $\half+it$. They proved
\begin{thm}[Diamond, Montgomery and Vorhauer, \cite{diamond}]
\label{thm:diamond}
 Let $\half<\theta<1$ and $a>(4/e)(1-\theta)$ be fixed. Then there is a system of Beurling primes $\B$ such that
 \begin{itemize}
  \item[(i)] $N_\B(x)=\kappa x + O(x^\theta)$ for $\kappa>0$;
  \item[(ii)] $\zeta_\B(s)$ is analytic for $\sigma\geq \theta$, except 
  for a simple pole at $s=1$ with residue $\kappa$;  
  \item[(iii)] $\zeta_\B(s)$ has infinitely many zeros on the curve $\sigma=1-a/\log t, t\geq 2,$ and no zeros 
   to the right of this curve;
   \item[(iv)] \[\sum_{p^k\leq x} \log p = x + O(x\exp(-2\sqrt{a\log x})).\]
 \end{itemize}
\end{thm}
From now onwards, $\B$ will denote a set of generalized Beurling primes that satisfy the conditions of the above 
theorem. 
Define the M\"obius function $\mu_\B$ associated to $\B$ by imitating the standard M\"obius function. 
Let $\mu_\B$ be a multiplicative function supported on the squarefree Beurling integers such that $\mu_\B(p_j)=-1$.
Let
\[\Delta(x)=\sum_{n\leq x}^*\mu_\B(n).\]
Given a small enough $\epsilon>0$ and a postive constant $C>0$, define
 \begin{align*}
 \A_1(C, \epsilon)&=\left\{x: \Delta(x)>Cx^{1-\epsilon}\right\},\\
 \A_2(C, \epsilon)&=\left\{ x : \Delta(x)<-Cx^{1-\epsilon}\right\}.\\
\end{align*}
Using Theorem~\ref{thm:omega_pm_main},
we prove that 
\[\m\left(\A_j(C, \epsilon)\cap[T, 2T]\right)=\Omega\left(T^{1-2\epsilon}\right) \text{ for } j=1, 2.\]

\section{Mellin Transform Of The Error Term}\label{sec:analytic_continuation}
The following assumptions on $D(s), \ \M(x)$ and $\Delta(x)$ which are as in (\ref{def:D}) and (\ref{def:M}), 
holds in wide generality.
\begin{asump}\label{as:for_continuation_mellintran}
Suppose there exist real numbers $T_0, \sigma_1, \sigma_2$ satisfying $0<\sigma_1<\sigma_2$ and $T_0>0$ 
such that
\begin{enumerate}
 \item[(i)]
$D(s)$ is absolutely convergent for $\Re(s)> \sigma_2$.
 \item[(ii)]
$D(s)$ can be meromorphically continued to the half plane $\Re(s)>\sigma_1$ and is analytic on the following line segments
\begin{align*}
& \{\sigma+it:\sigma_1\leq\sigma\leq\sigma_2, t=\pm T_0\}\\ 
& \{\sigma+it:\sigma=\sigma_1, -T_0\leq t\leq T_0\}.
\end{align*}

\item[(iii)]
For $\mathcal{P}$ defined as
\[\mathcal{P}=\{\sigma+it: \sigma+it \text{ is a pole of } D(s), \sigma>\sigma_1, |t|<T_0\},\]
the main term $\M(x)$ is sum of residues of $\frac{D(s)x^s}{s}$ at poles in $\mathcal P$:
 \[\M(x)=\sum_{\rho\in \mathcal{P}}\res_{s=\rho}\left( \frac{D(s)x^s}{s}\right).\] 
 
\end{enumerate}
\end{asump}

\begin{note}\label{note:initial_assumption_consequences}
We may also observe:
\begin{enumerate}
\item[(i)] The set $\mathcal{P}$ is finite.
\item[(ii)]  For any $\epsilon>0$, we have
  \[|a_n|, |\M(x)|, |\Delta(x)|, \left|\sum_{n\leq x}a_n\right| \ll x^{\sigma_2+\epsilon}. \]
\item[(iii)] The main term $\M(x)$ is a polynomial in $x$, and $\log x$:
  \[\M(x)=\sum_{j\in\mathscr{J}}\nu_{1, j}x^{\nu_{2, j}}(\log x)^{\nu_{3, j}},\]
  where $\nu_{1, j}$ are complex numbers, $\nu_{2, j}$ are real numbers with 
  $\sigma_1<\nu_{2, j}\leq \sigma_2$, $\nu_{3, j}$ are positive integers, 
  and $\mathscr J$ is a finite index set.
\end{enumerate}
\end{note}
\noindent
Meromorphic continuation of $A(s)$ has been obtained in special cases by many authors (see \cite{AnderOsci}).
We put this into a general setup in the following theorem. We skip the proof as it is routine and follows 
from expressing $A(s)$ as a contour integration involving $D(s)$ which is available from Perron
summation formula. For a detailed proof see \cite[II.2]{thesis1}. First, we define our required contour $\Co$.

\begin{defi}\label{def:contour}
Let $\sigma_1, \sigma_2$ and $T_0$ be as defined in Assumptions~\ref{as:for_continuation_mellintran}.
Choose a positive real number  $\sigma_3$ such that
$\sigma_3>\sigma_2.$
  We define the contour $\mathscr{C}$, as in Figure~1, as the union of the following five
  line segments: 
  \[\mathscr{C}=L_1\cup L_2\cup L_3 \cup L_4 \cup L_5 ,\]
  where
\begin{align*}
L_1=&\{\sigma_3+iv: T_0 \leq v < \infty\}, 
&L_2=\{u+iT_0: \sigma_1\leq u\leq \sigma_3 \}, \\
L_3=&\{\sigma_1+iv: -T_0 \leq v \leq T_0\}, 
&L_4=\{u-iT_0: \sigma_1\leq u\leq \sigma_3 \}, \\
L_5=&\{\sigma_3+iv: -\infty< v \leq -T_0\}. \\
\end{align*}
\end{defi}

\begin{center}
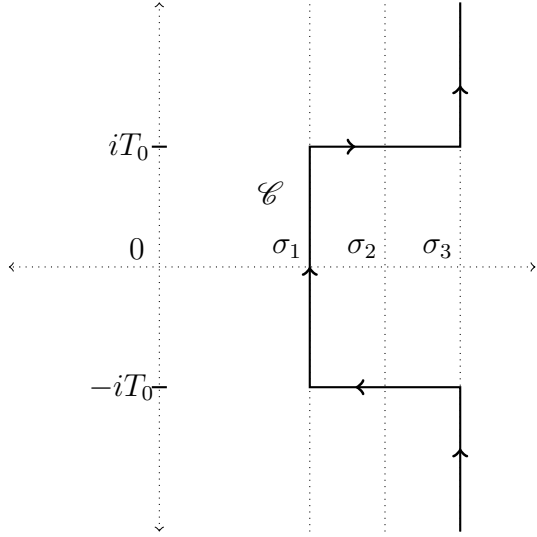
\begin{figure}\label{fg:contourc0}
 \begin{tikzpicture}[yscale=0.8]
\draw [<->][dotted] (0, -4.4)--(0, 4.4);
\node at (-0.4, 2) {$iT_0$}; 
\draw [thick] (-0.1, 2 )--(0.1,2);
\node at (-0.3, 0.3) {$0$};
\node at (-0.5, -2) {$-iT_0$};
\draw [thick] (-0.1,-2 )--(0.1, -2);
\draw [<->][dotted] (5, 0)--(-2, 0);

\draw [dotted] (2, -4.4)--(2, 4.4);
\node at (1.7, 0.3) {$\sigma_1$};
\draw [dotted] (3, -4.4)--(3, 4.4);
\node at (2.7, 0.3) {$\sigma_2$};
\draw [dotted] (4, -4.4)--(4, 4.4);
\node at (3.7, 0.3) {$\sigma_3$};

\draw [thick] [postaction={decorate, decoration={ markings,
mark= between positions 0.1 and 0.99 step 0.2 with {\arrow[line width=1.2pt]{>},}}}]
(4, -4.4)--(4, -2)--(2, -2)--(2, 2)--(4, 2)--(4, 4.4);
\node [below left] at (1.8, 1.6) {$\mathscr{C}$};
\end{tikzpicture}
\caption{Contour $\mathscr{C}$}
\end{figure}
\end{center}
In the above definition, the set of poles of $D(s)$ that lie to the right of $\Co$
is exactly the set $\mathcal{P}$.
The required analytic continuation of $A(s)$ is given as follows:
\begin{thm}\label{thm:analytic_continuation_mellin_transform}
Under the conditions in Assumptions-\ref{as:for_continuation_mellintran}, we have
\begin{equation}\label{eq:analytic_conti_mellin}
 A(s)=\twopi\int_{\Co}\frac{D(\eta)}{\eta(s-\eta)}\d\eta,
\end{equation}
when $s$ lies to the right hand side of the contour $\Co$.
\end{thm}
\section{The Oscillation Theorem Of Landau }\label{sec:preliminaries}
We begin with a criterion for functions that do not change sign.
This theorem appears in \cite{AnderOsci} and was attributed to Landau \cite{Landau}.
\begin{thm}[Landau\label{thm:landau_representation_integral}]
 Let $f(x)$ be a piecewise continuous function defined on $[1, \infty)$
 that does not change sign when $x>x_0$ for some $1<x_0<\infty$. Define 
 \[F(s):=\int_{1}^{\infty}\frac{f(x)}{x^{s+1}}\d x,\]
and assume that the above integral is absolutely convergent in some half plane.
Further, assume that we have an analytic continuation of $F(s)$ in a region containing the following part of the real line:
\[l(\sigma_0, \infty):=\{\sigma + i0: \sigma > \sigma_0\}.\]
Then the integral representing $F(s)$ is absolutely convergent for $\Re(s)>\sigma_0$, and hence $F(s)$ is an analytic function in this region. 
\end{thm}

We shall use Landau's theorem indirectly, by method of contradiction, to show the sign changes of $\Delta(x)$. 
 
Consider the Mellin transformation $A(s)$ of $\Delta(x)$. We need the following general assumptions
to apply Landau's theorem.

\begin{asump}\label{as:for_landau}
 Suppose there exists a real number $\sigma_0$
 such that $A(s)$ has the following properties:
\begin{itemize}
\item[(i)] There exists $t_0\neq 0$ such that 
 \begin{equation*}
 \lambda:=\limsup\limits_{\sigma\searrow\sigma_0}(\sigma-\sigma_0)|A(\sigma+it_0)|>0 .
 \end{equation*}
\item[(ii)] We have 
\begin{align*}
l_s&:=\limsup\limits_{\sigma\searrow\sigma_0}(\sigma-\sigma_0)A(\sigma) < \infty,\\
l_i&:= \liminf\limits_{\sigma\searrow\sigma_0}(\sigma-\sigma_0)A(\sigma) > -\infty.
\end{align*}
\item[(iii)]The limits $l_i, l_s$ and $\lambda$ satisfy
\[l_i+\lambda>0\quad \text{and}\quad l_s-\lambda<0.\]
\item[(iv)] We can analytically continue $A(s)$ in a region containing the real line
$l(\sigma_0, \infty)$.
\end{itemize} 
\end{asump}

\begin{rmk}
Assumptions~\ref{as:for_landau} (i) implies that $\sigma_0+it_0$ is a singularity of $A(s)$.
\end{rmk}
We construct the following sets for further use.
\begin{defi}\label{def:A1A2}
With $l_s, l_i$ and $\lambda$ as in Assumptions~\ref{as:for_landau}, and for an $\epsilon$ such that $
0<\epsilon<\min(\lambda+l_i, \lambda-l_s)$, define
 \begin{align*}
 &\mathcal{A}_1:=\{x:x \in [1, \infty), \Delta(x)>(l_i+\lambda-\epsilon)x^{\sigma_0}\}\\
\text{ and }\quad&\mathcal{A}_2:=\{x:x \in [1, \infty), \Delta(x)<(l_s-\lambda+\epsilon)x^{\sigma_0}\}.
\end{align*}
\end{defi}

\subsection{$\Omega_\pm$ Results}
Under Assumptions~\ref{as:for_landau} and using methods from \cite{KaczMeasure}, we can derive the following measure theoretic theorem.

\begin{thm}\label{thm:landu_omegapm}
Let the conditions in Assumptions~\ref{as:for_landau} hold. Then for any 
 real number $M>1$, we have
 \begin{align*}
 \m(\mathcal{A}_j\cap[M, \infty))>0, \ \text{ for } j=1, 2.
 \end{align*}
This implies
 \begin{align*}
\Delta(x)=\Omega_\pm(x^{\sigma_0}). 
\end{align*}
\end{thm}
 
 \begin{proof}
We prove the theorem only for $\mathcal{A}_1$ as the other part is similar.
 
Now define the following integrals whenever they are absolutely convergent: 
 \begin{align*}
  &g(x):=\Delta(x)-(l_i+\lambda-\epsilon)x^{\sigma_0}, \quad &&G(s):=\int_{1}^{\infty}\frac{g(x)}{x^{s+1}}{\d x};\\
  &g^+(x):=\max(g(x), 0), \quad &&G^+(s):= \int_{1}^{\infty}\frac{g^+(x)}{x^{s+1}}{\d x};\\
  &g^-(x):=\max(-g(x), 0), \quad &&G^-(s):= \int_{1}^{\infty}\frac{g^-(x)}{x^{s+1}}{\d x}.
 \end{align*}
With the above notations, we have
\begin{align*}
 &g(x)=g^+(x)-g^-(x) \\
 \text{ and }\quad &G(s)=G^+(s)-G^-(s).
\end{align*}
Note that
\begin{align*}
G(s)&=A(s) - \int_1^{\infty}(l_i+\lambda-\epsilon)x^{\sigma_0-s-1}{\d x} \\
&=A(s) + \frac{l_i+\lambda-\epsilon}{\sigma_0-s}\quad \text{for }  \Re(s)> \sigma_0,
\end{align*}
where $\epsilon$ is fixed as in Definition \ref{def:A1A2}.
So $G(s)$ is analytic wherever $A(s)$ is, except possibly for a pole at $\sigma_0$.
This gives
\begin{align}\label{eq:G_lim_sigma0t0}
 \limsup\limits_{\sigma\searrow\sigma_0}(\sigma-\sigma_0)|G(\sigma+it_0)|
 =\limsup\limits_{\sigma\searrow\sigma_0}(\sigma-\sigma_0)|A(\sigma+it_0)|=\lambda.
\end{align}
We shall use the above limit to prove our theorem. 
We proceed by method of contradiction. 
Assume that there exists an $M>1$ such that 
\[\m(\mathcal{A}_1\cap[M, \infty))=0.\] 
This implies
\[G^+(s)=\int_{1}^{\infty}\frac{g^+(x)}{x^{s+1}}{\d x} = \int_{1}^{M}\frac{g^+(x)}{x^{s+1}}{\d x} \] 
is bounded for any $s$, and so is an entire function.
By Assumptions~\ref{as:for_landau}, 
$A(s)$ and $G(s)$ can be analytically continued on the line $l(\sigma_0, \infty)$.
As $G(s)$ and $G^+(s)$ are analytic on $l(\sigma_0, \infty)$, $G^-(s)$ is also analytic on $l(\sigma_0, \infty)$.
The integral for $G^-(s)$ is absolutely convergent for $\Re(s)>\sigma_3+1$, and $g^-(x)$ is a piecewise continuous function. 
This suggests that we can apply Theorem~\ref{thm:landau_representation_integral} to 
$G^-(s)$, and conclude that 
\[G^-(s)=\int_{1}^{\infty}\frac{g^-(x)}{x^{s+1}}{\d x} \]
is absolutely convergent for $\Re(s)>\sigma_0$.

From the above discussion, we summarize that the Mellin transforms of $g, g^+$ and $g^-$ converge absolutely 
for $\Re(s)>\sigma_0$. As a consequence, we see that $G(\sigma), G^+(\sigma)$ and $G^-(\sigma)$ 
are finite real numbers for $\sigma>\sigma_0$. We note that for any $t\in \R$
\[ |G^+(\sigma+it)| \le \int_1^{M}\frac{g^+(x)}{x^{\sigma+1}}\d x =O(1).\]
Thus 
$$(\sigma-\sigma_0)|G^+(\sigma+it)| \longrightarrow 0 \text{ as } \sigma \longrightarrow \sigma_0+ .$$
Observe that
\begin{align*}
(\sigma-\sigma_0)|G(\sigma+ it_0)|&\leq (\sigma-\sigma_0)G^+(\sigma)+ (\sigma-\sigma_0)G^-(\sigma)\\
&\leq 2(\sigma-\sigma_0)G^+(\sigma)- (\sigma-\sigma_0)G(\sigma)\\
&\leq 2(\sigma-\sigma_0)G^+(\sigma) -(\sigma-\sigma_0)A(\sigma) + l_i+\lambda-\epsilon.
\end{align*}
%
So we have
\begin{equation*}
 \limsup\limits_{\sigma\searrow\sigma_0}(\sigma-\sigma_0)|G(\sigma+it_0)|
 \leq - \liminf\limits_{\sigma\searrow\sigma_0}(\sigma-\sigma_0)A(\sigma)+l_i+\lambda-\epsilon 
 = \lambda-\epsilon.
\end{equation*}
This contradicts (\ref{eq:G_lim_sigma0t0}). 
Thus $\m(\mathcal{A}_1\cap[M, \infty))>0$ for any $M>1$, which completes the proof.
\end{proof}

\subsection{Measure Theoretic $\Omega_\pm$ Results}
Results in the last section shows that $\mathcal{A}_1$ and $\mathcal{A}_2$ 
are unbounded. But they do not indicate how the size of these sets $\mathcal{A}_1$ and $\mathcal{A}_2$ grow. 
An answer to this question was given by Kaczorowski and Szyd\l o in \cite[Theorem~4]{KaczMeasure}.

\begin{thm}[Kaczorowski and Szyd\l o \cite{KaczMeasure}]\label{thm:kaczorowski_correct}
Let the conditions in Assumptions~\ref{as:for_landau} hold.
Also assume that for a non-decreasing positive continuous function $h$ satisfying
\[h(x)\ll x^{\epsilon},\]
we have
\begin{equation}\label{eq:normDelta}
\int_{T}^{2T}\Delta^2(x){\d x}\ll T^{2\sigma_0 + 1}h(T).
\end{equation}
Then as $T\rightarrow \infty$,
\[\m\left( \mathcal{A}_j\cap[1, T]\right)=\Omega\left(\frac{T}{h(T)}\right)\quad \text{ for } j=1, 2.\]
\end{thm}
In \cite{gautami},  Bhowmik, Ramar\'e and Schlage-Puchta proved a more localized version of this theorem.

In the above theorem, (\ref{eq:normDelta}) is a very strong condition to hold. 
We generalize Theorem~\ref{thm:kaczorowski_correct} below, 
where we do not require the condition $h(x)\ll x^\epsilon$.

\begin{thm}\label{thm:omega_pm_main}
Let the conditions in Assumptions~\ref{as:for_landau} hold.
Assume that $A(s)$ has analytic continuation in a region containing the real line $l(\sigma_0, \infty)$.
Let $h_1$ and $h_2$ be two positive functions
such that
\begin{equation}\label{eq:second_moment_error}
\int_{[T, 2T]\cap \mathcal{A}_j}
\frac{\Delta^2(x)}{x^{2\sigma_0+1}}\d x\ll h_j(T)\quad\text{ for } j=1, 2.
\end{equation}
Then as $T \longrightarrow \infty$, 
\begin{equation}\label{eq:omega_pm_measure}
 \m(\mathcal{A}_j\cap[T, 2T])=\Omega\left(\frac{T}{h_j(T)}\right)\quad\text{ for } j=1, 2.
\end{equation}
\end{thm}
\begin{proof}
 We prove the theorem for the measure of $\mathcal{A}_1$; the proof is similar for $\mathcal{A}_2$ .
We define $g, g^+, g^-, G, G^+$ and $G^-$, as in Theorem~\ref{thm:landu_omegapm} of
Section~\ref{sec:preliminaries}.
Assume that
\begin{equation}\label{eq:measure_contradic_A1}
\m\left(\mathcal{A}_1\cap[T, 2T]\right)=o\left(\frac{T}{h_1(T)}\right).
\end{equation}
This implies that for any $\epsilon'>0$, there exists an integer 
$k(\epsilon')>0$ such that 
\begin{equation}\label{eq:bound_from_littleo_assumption}
\frac{h_1(2^k)\m(\mathcal{A}_1\cap[2^k, 2^{k+1}])}{2^k}<\epsilon'^2,
\end{equation}
for all $k>k(\epsilon')$.
Using the above assumption, we may obtain an upper bound for $G^+(\sigma)$ as follows:
\begin{align*}
& \int_{\A_1}\frac{g^+(x)\d x}{x^{\sigma+1}}
\leq \sum_{k\geq 0}\int_{\A_1\cap[2^k, 2^{k+1}]}\frac{\Delta(x)\d x}{x^{\sigma+1}}\\
&( \text{as } \Delta(x)\geq g(x) \text{ on } \A_1 \text{ by Assumptions~\ref{as:for_landau},(iii)})\\
&\leq \sum_{k\geq 0}\left(\int_{\A_1\cap[2^k, 2^{k+1}]}\frac{\Delta^2(x)\d x}{x^{2\sigma_0+1}}\right)^{1/2}
\left(\frac{\m(\A_1\cap[2^k, 2^{k+1}])}{2^{k(2(\sigma-\sigma_0)+1)}}\right)^{1/2}\\
&\leq c_2 \sum_{k\geq 0}\left(\frac{h_1(2^k)\m(\A_1\cap[2^k, 2^{k+1}])}{2^{k(2(\sigma-\sigma_0)+1)}}\right)^{1/2}\\
&\leq c_3(\epsilon') + \epsilon' \sum_{k\geq k(\epsilon')}\frac{1}{2^{k(\sigma-\sigma_0)}}  \quad (\text{Using } 
(\ref{eq:bound_from_littleo_assumption})).\\
\end{align*}
In the above inequalities, $c_2$ and $c_3(\epsilon')$ are some positive constants, and $c_3(\epsilon')$ depends on $\epsilon'$.
We summarize the above calculation to
\begin{equation}\label{eq:bound_G+}
G^+(\sigma)\leq c_3(\epsilon') + \epsilon'\frac{2^{-(\sigma-\sigma_0)}}{(\sigma-\sigma_0)\log2} . 
\end{equation}
Therefore
\[G^+(s)=\int_1^{\infty}\frac{g^+(x)\d x}{x^{s+1}}\]
is absolutely convergent for $\Re(s)>\sigma_0$, and is analytic in this region.
But
\[G^-(s)=G(s)-G^+(s),\]
and $G$ is analytic on $l(\sigma_0, \infty)$. So $G^-$ is also analytic on $l(\sigma_0, \infty)$. 
Using Theorem~\ref{thm:landau_representation_integral}, we get 
\[ G^-(s)=\int_1^{\infty}\frac{g^-(x)\d x}{x^{s+1}}\]
is also absolutely convergent for $\Re(s)>\sigma_0$. As a consequence, we get $G(\sigma), G^+(\sigma)$ and $G^-(\sigma)$
are real numbers for $\sigma>\sigma_0$.

Now observe that
\begin{align*}
(\sigma-\sigma_0)|G(\sigma+ it_0)|&\leq 2(\sigma-\sigma_0)G^+(\sigma)- 
(\sigma-\sigma_0)G(\sigma )\\
&\leq 2(\sigma-\sigma_0)G^+(\sigma) -(\sigma-\sigma_0)A(\sigma) + l_i+\lambda-\epsilon.
\end{align*}
%
Using (\ref{eq:bound_G+}), 
\begin{align*}
&(\sigma-\sigma_0)|G(\sigma+it)|\\
&\le 2(\sigma-\sigma_0)c_3(\epsilon') + \frac{2\epsilon'}{2^{\sigma-\sigma_0}\log2}
-(\sigma-\sigma_0)A(\sigma) + l_i+\lambda-\epsilon.
\end{align*}
Choosing $\frac{2\epsilon'}{\log2}<\frac{\epsilon}{2}$, we have
\[\limsup\limits_{\sigma\searrow\sigma_0}(\sigma-\sigma_0)|G(\sigma+it_0)|<\lambda-\frac{\epsilon}{2},\]
which contradicts (\ref{eq:G_lim_sigma0t0}). Thus (\ref{eq:measure_contradic_A1}) is false, and 
\begin{equation}
\m\left( \mathcal{A}_1\cap[x, 2x)\right)=\Omega\left(\frac{T}{h_1(T)}\right).
\end{equation}
\end{proof}

The method of proof of the last theorem suggests the following.

\begin{thm}\label{coro:omega_pm_secondmoment}
Let the conditions in Assumptions~\ref{as:for_landau} hold.
 Assume that there is an 
analytic continuation of $A(s)$ in a region containing the real line $l(\sigma_0, \infty)$. Then we have
 \begin{equation}\label{eq:omega_pm_secondmoment}
 \int_{[T, 2T]\cap \A_j}\Delta^2(x)\d x = \Omega(T^{2\sigma_0 + 1}) \quad \text{ for } j=1, 2. 
 \end{equation}
\end{thm}
\begin{proof}
We shall prove this for $\A_1$, and the proof for $\A_2$ is similar.  
Note that as an important part of the proof of Theorem~\ref{thm:omega_pm_main}, we showed that the integral for $G^+(s)$ 
is absolutely convergent for $\Re(s)>\sigma_0$, by assuming (\ref{eq:omega_pm_measure}) is false. Then we got a contradiction that 
proves (\ref{eq:omega_pm_measure}). Now we proceed in a similar manner 
by assuming (\ref{eq:omega_pm_secondmoment}) is false. So we have
\begin{equation}\label{eq:second_moment_contra}
 \int_{[T, 2T]\cap \A_1}\Delta^2(x)\d x = o(T^{2\sigma_0 + 1}). 
\end{equation}
So for an arbitrarily small constant $\epsilon'$, we have 
\begin{align*}
& |G^+(s)|\leq \int_{\A_1}\frac{g^+(x)\d x}{x^{\sigma+1}}
\leq \sum_{k\geq 0}\int_{\A_1\cap[2^k, 2^{k+1}]}\frac{\Delta(x)\d x}{x^{\sigma+1}}\\
&\leq \sum_{k\geq 0}\frac{1}{2^{k(\sigma-\sigma_0)}}\left(\int_{\A_1\cap[2^k, 2^{k+1}]}\frac{\Delta^2(x)\d x}{x^{2\sigma_0+1}}\right)^{1/2}\\
&\leq c_4(\epsilon') + \epsilon'\sum_{k\geq k(\epsilon')}\frac{1}{2^{k(\sigma-\sigma_0)}}, \\
\end{align*}
where $c_4(\epsilon')$ is a positive constant depending on $\epsilon'$. From this, we obtain that $G^+(s)$ is absolutely convergent
for $\Re(s)>\sigma_0$. 
Further, arguments similar to the proof of Theorem~\ref{thm:omega_pm_main} yield a contradiction to (\ref{eq:second_moment_contra}).
\end{proof}

\subsection{Applications}
\subsubsection{The Twisted Divisor Function}
We define $\Delta(x)$ as in (\ref{eq:formmula_tau_ntheta}).
An upper bound for $\Delta(x)$ (as in (\ref{eq:upper_bound_delta})) can be computed using Perron's formula
and fourth moment estimates of $\zeta(s)$ at $\Re(s)=\frac{1}{2}$ ( see \cite[Theorem~33 ]{DivisorsHallTenen} ). 
Define a contour $\Co$ as given in Figure~\ref{fg:contour_c0_taun}:
\begin{align*}
 \Co=&\left(\frac{5}{4}-i\infty, \frac{5}{4}-i(\theta+1)\right]\cup \left[\frac{5}{4}-i(\theta+1), \frac{3}{4}-i(\theta+1)\right]\\
 &\cup \left[\frac{3}{4}-i(\theta+1), \frac{3}{4}+i(\theta+1)\right]\cup \left[\frac{3}{4}+i(\theta+1), \frac{5}{4}+i(\theta+1)\right]\\
 &\cup \left[\frac{5}{4}+i(\theta+1), \frac{5}{4}+i\infty \right).
\end{align*}
From Perron's formula,
we can show that
\[\Delta(x)=\twopi\int_{\Co}\frac{D(\eta)x^\eta}{\eta}\d \eta.\]
Using Theorem~\ref{thm:analytic_continuation_mellin_transform}, we have 
\[A(s)=\twopi\int_{1}^{\infty}\frac{\Delta(x)}{x^{s+1}}\d x=\twopi\int_{\Co}\frac{D(\eta)}{\eta(s-\eta)}\d \eta,\]
when $s$ lies to the right of the contour $\Co$. Denote the first nontrivial zero of $\zeta(s)$ with least positive imaginary part by $2s_0$, 
which is approximately
\begin{equation}\label{s0}
2s_0=\frac{1}{2}+i14.134\ldots .
\end{equation}
Define the contour $\Co(s_0)$, as in Figure~\ref{fg:contour_c0s0}, such that $s_0$ and any real number $s>1/4$ lie to the 
right of this contour. A meromorphic continuation of $A(s)$ to all $s$ that lies to the right of $\Co(s_0)$ is given by
\begin{equation}\label{eq:analytic_cont1/4}
 A(s)=\twopi\int_{\Co(s_0)}\frac{D(\eta)}{\eta(s-\eta)}\d \eta + \frac{\underset{\eta=s_0}{\res}D(\eta)}{s_0(s-s_0)}.
\end{equation}
\begin{center}
 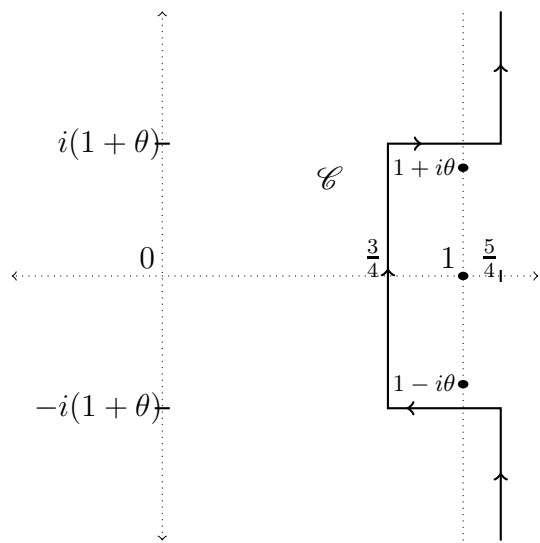
\begin{figure}
 \begin{tikzpicture}[yscale=0.8]
\draw [<->][dotted] (-1, -4.4)--(-1, 4.4);
\node at (-1.2, 0.3) {$0$};
\draw [<->][dotted] (4, 0)--(-3, 0);

\draw [dotted] (3, -4.4)--(3, 4.4);
\fill (3, 0) circle[radius=2pt];
\node at (2.8, 0.3) {$1$};

\fill (3, 1.8) circle[radius=2pt];
\node at (2.48, 1.8) {\scriptsize{$1+i\theta$}};
\fill (3, -1.8) circle[radius=2pt];
\node at (2.48, -1.8) {\scriptsize{$1-i\theta$}};
\node at (1.8, 0.3) {$\frac{3}{4}$};
\node at (3.35, 0.3) {$\frac{5}{4}$};
\draw [thick] (3.5,-0.1 )--(3.5, 0.1);
\node at (-1.85, -2.2) {$-i(1+\theta)$};
\draw [thick] (-1.1,-2.2)--(-0.9, -2.2);
\node at (-1.7, 2.2) {$i(1+\theta)$};
\draw [thick] (-1.1,2.2)--(-0.9, 2.2);
\draw [thick] [postaction={decorate, decoration={ markings,
mark= between positions 0.09 and 0.98 step 0.21 with {\arrow[line width=1pt]{>},}}}]
(3.5, -4.4)--(3.5, -2.2)--(2, -2.2)--(2, 2.2)--(3.5, 2.2)--(3.5, 4.4);
\node [below left] at (1.55, 2) {$\mathscr{C}$};
\end{tikzpicture}
\caption{Contour $\mathscr{C}$, for $D(s)=\sum_{n=1}^{\infty}\frac{|\tau(n, \theta)|^2}{n^s}$.}\label{fg:contour_c0_taun}
\end{figure}
\end{center}
From (\ref{eq:analytic_cont1/4}) we calculate the following two limits:
\begin{equation}\label{eqn:lambda_theta}
\lambda(\theta):= \lim_{\sigma\searrow0}\sigma|A(\sigma+s_0)| = |s_0|^{-1} \left|\underset{\eta=s_0}{\res}D(\eta)\right|>0,
\end{equation}
and 
\begin{equation*}
\lim_{\sigma\searrow0}\sigma A(\sigma+ 1/4)=0.
\end{equation*}
For a fixed small enough $\epsilon>0$, define
\begin{align*}
 \A_1&=\left\{x: \Delta(x)>(\lambda(\theta)-\epsilon)x^{1/4}\right\}\\
 \text{and} \quad \A_2&=\left\{ x : \Delta(x)<(-\lambda(\theta)+\epsilon)x^{1/4}\right\}.
\end{align*}
\begin{center}
 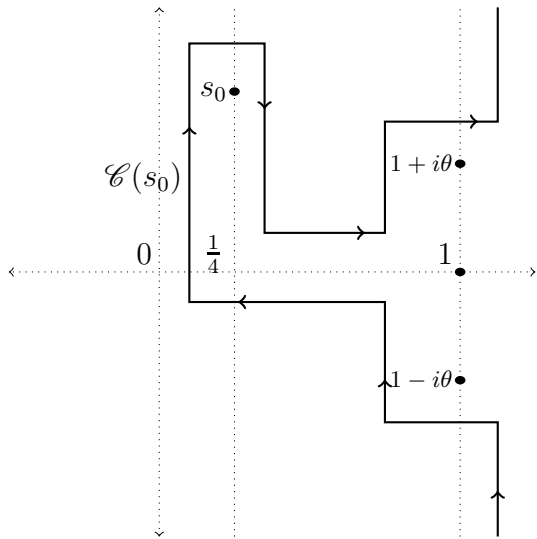
\begin{figure}
 \begin{tikzpicture}[yscale=0.8]
\draw [<->][dotted] (-1, -4.4)--(-1, 4.4);
\node at (-1.2, 0.3) {$0$};
\draw [<->][dotted] (4, 0)--(-3, 0);

\draw [dotted] (3, -4.4)--(3, 4.4);
\fill (3, 0) circle[radius=2pt];
\node at (2.8, 0.3) {$1$};

\fill (3, 1.8) circle[radius=2pt];
\node at (2.48, 1.8) {\scriptsize{$1+i\theta$}};
\fill (3, -1.8) circle[radius=2pt];
\node at (2.48, -1.8) {\scriptsize{$1-i\theta$}};
\draw [dotted] (0, -4.4)--(0, 4.4);
\node at (-0.28, 0.3) {$\frac{1}{4}$};
\fill (0, 3) circle[radius=2pt];
\node at (-0.28, 3) {$s_0$};

\draw [thick] [postaction={decorate, decoration={ markings,
mark= between positions 0.03 and 0.99 step 0.147 with {\arrow[line width=1pt]{>},}}}]
(3.5, -4.4)--(3.5, -2.5)--(2, -2.5)--(2, -0.5)--(-0.6, -0.5)--(-0.6, 3.8)--(0.4, 3.8)--(0.4, 0.65)
--(2, 0.65)--(2, 2.5)--(3.5, 2.5)--(3.5, 4.4);
\node [below left] at (-0.55, 2) {$\mathscr{C}(s_0)$};
\end{tikzpicture}
\caption{Contour $\mathscr{C}(s_0)$}\label{fg:contour_c0s0}
\end{figure}
\end{center}
Upper-bound of $\Delta$ from (\ref{eq:upper_bound_delta}) and 
Theorem~\ref{thm:omega_pm_main} give
\begin{equation}\label{result:tau_n_theta_omegapm}
 \m\left(\A_j\cap[T, 2T]\right)=\Omega\left(T^{1/2}(\log T)^{-12}\right) \text{ for } j=1, 2.
\end{equation}
Note that the above statements in particular show that
\[\Delta(x)=\Omega_\pm\left(x^{1/4}\right).\]
From Theorem~\ref{coro:omega_pm_secondmoment}, we get
\begin{equation}\label{result:tau_n_theta_secondmoment}
 \int_{\A_j\cap[T, 2T]}\Delta^2(x)\d x = \Omega\left(T^{3/2}\right) \text{ for } j=1, 2.
\end{equation}

\subsubsection{The M\"obius Function}
\label{subsec:mobius}     
Let 
 \[\Delta(x)=\sum_{n\leq x}^*\mu(n)\]
and 
\[\lambda=|s_0|^{-1} \left|\underset{\eta=s_0}{\res}\zeta^{-1}(\eta)\right|,\]
where $s_0$ is the first nontrivial zero of $\zeta(s)$.
Recall that
\begin{align*}
\A_1&=\left\{x: \Delta(x)>(\lambda-\epsilon_0)x^{1/2}\right\}\\
\text{and} \quad \A_2&=\left\{ x : \Delta(x)<(-\lambda+\epsilon_0)x^{1/2}\right\},
\end{align*}
for a fixed $\epsilon_0$ such that $0<\epsilon_0<\lambda_2$. 
We shall show
\begin{equation}\label{res:mu}
 \m\left(\A_j\cap[T, 2T]\right)=\Omega\left(T^{1-\epsilon}\right),
\end{equation}
for  $j=1, 2$ and for any $\epsilon>0$.

Here we apply Theorem~\ref{thm:omega_pm_main} in a similar way as in the previous application, so we shall skip the details. 
Under Riemann Hypothesis, Theorem~\ref{thm:kaczorowski_correct} gives
\[\m\left(\A_j\cap[T, 2T]\right)=\Omega\left(T \exp\left(-(\log x)(\log\log x)^{5+\epsilon}\right)\right) \text{ for } j=1, 2.\]
which implies (\ref{res:mu}). 

But if the Riemann Hypothesis is false, 
there exists  a constant $\mathfrak{a}$, with $1/2<\mathfrak{a}\leq1$, such that
\[\mathfrak{a}=\sup\{\sigma:\zeta(\sigma+it)=0\}.\]
Using Perron summation formula, we may show that
\[\Delta(x)\ll x^{\mathfrak{a}+\epsilon} \quad \text{ for any } \epsilon>0.\]
Also for any arbitrarily small $\delta$, we have $\mathfrak{a}-\delta<\sigma'\leq\mathfrak{a}$ such that
 $\zeta(\sigma'+it')=0$ for some real number $t'$. If $\lambda'':=|\sigma'+it'|^{-1}\left|\underset{\eta=\sigma'+it'}{\res}\zeta^{-1}(\eta)\right|$, 
 then by Theorem~\ref{thm:omega_pm_main} we get 
\begin{align*}
 \m\left(\left\{x\in[T, 2T]:\Delta(x)>(\lambda''/2)x^{\sigma'}\right\}\right)&=\Omega\left(T^{1-2\delta-2\epsilon}\right)\\
 \text{ and } \quad \m\left(\left\{x\in[T, 2T]:\Delta(x)<-(\lambda''/2)x^{\sigma'}\right\}\right)&=\Omega\left(T^{1-2\delta-2\epsilon}\right).
\end{align*}
As $\delta$ and $\epsilon$ are arbitrarily small and $\sigma'>1/2$, the above $\Omega$ bounds imply (\ref{res:mu}).

\subsubsection{M\"obius Function of Beurling Primes}
Let $\Delta(x)=\sum_{n\leq x}^*\mu_{\B}(x)$.  We may note that trivially $\Delta(x)\leq x$. Let $\epsilon$ be a fixed positive constant.
By Theorem~\ref{thm:diamond}, there exists  $t_0>0$ such that
\[\zeta_\B\left(1-\frac{a}{\log t_0}+i_0t\right)=0 \text{ and } \frac{a}{\log t_0}<\epsilon.\]
We define 
\begin{align*}
\A_1'&=\left\{x: \Delta(x)>(\lambda/2)x^{1-\frac{a}{\log t_0}}\right\}\\
\text{and} \quad \A_2'&=\left\{ x : \Delta(x)<(-\lambda/2)x^{1-\frac{a}{\log t_0}}\right\},
\end{align*}
where $\lambda=\left|1-\frac{a}{\log t_0}\right|^{-1} \left|\underset{\eta=1-a/\log t_0}{\res}\zeta^{-1}(\eta)\right|$.
From Theorem~\ref{thm:omega_pm_main},
\[\m\left(\A_j'\cap[T, 2T]\right)=\Omega\left(T^{1-\frac{2a}{\log t_0}}\right).\]
Recall
\begin{align*}
 \A_1(C, \epsilon)&=\left\{x: \Delta(x)>Cx^{1-\epsilon}\right\},\\
 \A_2(C, \epsilon)&=\left\{ x : \Delta(x)<-Cx^{1-\epsilon}\right\},\\
\end{align*}
for a constant $C>0$. Since 
\[\A_j'\subset \A_j(C, \epsilon),\]
we have
\[\m\left(\A_j(C, \epsilon)\cap[T, 2T]\right)=\Omega\left(T^{1-2\epsilon}\right).\]

\subsection*{Acknowledgment} 
We thank G. Bhowmik for initiating us on this topic and providing us with a preprint of 
\cite{gautami}. We thank J. C. Schlage-Puchta for his insightful comments. 
We thank A. Ivi\'c, O. Ramar\'e and K. Srinivas for carefully reading the manuscript and for their suggestions, which 
made the paper more readable and up to date. KM is supported by Grant 227768 of the Research Council of Norway.


\end{document}